\documentclass[reqno,12pt]{amsart}
\usepackage{amsmath,amsthm}
\usepackage{amsfonts,euscript}
\usepackage{multirow,graphicx}
\def\C{\mathbb{C}}
\def\CP{\mathbb{C}P}
\def\R{\mathbb{R}}
\def\RP{\mathbb{R}P}

\def\calH{\mathcal H}
\def\calU{\mathcal U}
\def\calV{\mathcal V}
\def\calW{\mathcal W}

\def\ve{\mathbf e}
\def\w{\omega}
\def\wo{\overline{\omega}}
\def\ri{\mathrm{i}}
\def\Atilde{\widetilde{A}}

\def\Phat{\widehat{P}}
\def\Qhat{\widehat{Q}}
\def\II{\operatorname{II}}
\def\tr{\operatorname{tr}}

\def\rJ{\mathsf{J}}
\def\rA{\mathsf{A}}
\def\rAtilde{\widetilde{\rA}}
\def\Mhat{\widehat{M}}
\def\F{\EuScript{F}}
\def\ehat{\hat{\ve}}
\def\Ahat{\widehat{A}}
\def\frakB{\mathfrak{B}}
\def\I{\mathcal I}
\def\J{\EuScript{J}}
\def\ve{\mathbf{e}}
\renewcommand{\Re}{\operatorname{Re}}
\renewcommand{\Im}{\operatorname{Im}}

\theoremstyle{definition}
\newtheorem{theorem}{Theorem}

\newtheorem{cor}[theorem]{Corollary}
\newtheorem{prop}[theorem]{Proposition}
\newtheorem*{theorem*}{Theorem}
\newtheorem{conj}[theorem]{Conjectures}
\newtheorem{defn}{Definition}

\theoremstyle{remark}
\newtheorem{remark}[theorem]{Remark}
\newtheorem*{remark*}{Remark}

\begin{document}
\title{Stark Hypersurfaces in Complex Projective Space}
\author{Thomas A. Ivey}
\address{Dept. of Mathematics, College of Charleston, 66 George St., Charleston SC 29424}

\begin{abstract}
Stark hypersurfaces are a special class of austere hypersurface in $\CP^n$ where
the shape operator is compatible with the $CR$-structure.  In this paper, the possible shape operators
for stark hypersurfaces are completely determined, and stark hypersurfaces in $\CP^2$ are
constructed as integrals of a Frobenius exterior differential system.
\end{abstract}

\date{August  2015}
\maketitle

\section{Introduction}
Stark hypersurfaces are a special case of {\em austere submanifolds}.  The motivation for studying austere
submanifolds comes from the subject of calibrated geometry (in particular, special Lagrangian submanifolds) which
was pioneered by Harvey and Lawson in the early 1980s \cite{HL}.

On any K\"ahler manifold $X^{2n}$ the K\"ahler form and its powers (suitably normalized) are
calibrations, but on a Ricci-flat K\"ahler manifold the real part of the holomorphic
volume form $\Theta$ is also a calibration.  A real $n$-dimensional submanifold is {\em special Lagrangian}
if it calibrated by $\Theta$.
More generally, $\Re( e^{\ri \phi} \Theta)$, where angle $\phi$ is constant, is also a calibration,
and a submanifold $L$ is said to be special Lagrangian {\em with phase $e^{\ri \phi}$} if it is calibrated by this $n$-form.
Harvey and Lawson showed that this is equivalent to $L$ being Lagrangian and $\Im( e^{\ri \phi} \Theta)\vert_L = 0$.

It's easiest to understand the Lagrangian part of this condition in the flat case, where $X=\C^n$.  If we identify
$\C^n$ with $T\R^n$ (so that the zero section is the real slice and the fibers are tangent to imaginary directions), then the K\"ahler form is (up to sign)
the exterior derivative of the canonical form on the tangent bundle.  Thus, if $M \subset \R^n$ is any submanifold,
then its normal bundle $NM \subset T \R^n$ is automatically Lagrangian.  Harvey and Lawson calculated that
$NM$ is special Lagrangian (with a phase depending on $n$ and the dimension of $M$) if and only if $M$ is {\em austere},
i.e., all odd degree elementary symmetric functions of the eigenvalues of the second fundamental form of $M$, in any normal direction,
vanish.  (Equivalently, the eigenvalues are symmetrically arranged around zero on the real line.)
When $n$ is even examples of austere submanifolds are easy to generate; for, if $M$ is a holomorphic submanifold
of $\R^n \cong \C^{n/2}$ then its second fundamental form satisfies
$$\II(X, \rJ Y) = \II(\rJ X,Y),$$
where $\rJ$ is the complex structure.  Then the austere condition is automatic, since if $X$ is an eigenvector for $\nu \cdot \II$
then $\rJ X$ is an eigenvector for the opposite eigenvalue.

Some results have been obtained in the classification of austere submanifolds of Euclidean space: Bryant \cite{Baustere} classified
those of dimension 2 and 3, and provided models for the second fundamental forms in dimension 4; Dajczer and Florit \cite{DF} classified austere submanifolds
of arbitrary dimension but with Gauss map of rank 2.  Ionel and I \cite{II1} produced a partial classification in dimension 4, and classified
austere 4-folds that are ruled by 2-planes \cite{II2}.  But in this paper we are concerned with austere submanifolds in the curved ambient
space $\CP^n$, where the tangent bundle also carries a Ricci-flat K\"ahler metric due to Stenzel \cite{Stpaper}.

\section{The Austere Conditions}
In his thesis, Stenzel constructed a complete cohomogeneity-one Ricci-flat K\"ahler metric on the tangent space of any compact rank one symmetric space $X$.
We will say that a submanifold $M \subset X$ is {\em austere} if its normal bundle is
special Lagrangian with respect to the Stenzel metric.  This naturally leads to the question of what conditions
austerity imposes on the second fundamental form of $M$ in these geometries.  The first result in this direction was obtained
by Karigiannis and Min-Oo \cite{KM}, for the Stenzel metric on $T S^n$.  They showed that
for any submanifold $M \subset S^n$, the normal bundle $NM$ is always Lagrangian (a non-trivial result,
since Stenzel's K\"ahler form is not the same as the usual symplectic form on the tangent bundle), and
that $M$ is austere if and only if satisfies the Euclidean austere conditions (i.e., the odd degree
symmetric functions of the eigenvalues of $\II$ vanish).

Since the Stenzel metric on $T \R P^n$ is obtained from that on $T S^n$ by quotienting by the antipodal map,
the next space to be investigated is $\CP^n$.  In a work completed in 2014, Ionel and I obtained the following
conditions:
\begin{theorem}[\cite{II3}]  Let $M^k \subset \CP^n$ be a smooth submanifold, let $\nu$ denote an arbitrary
unit normal vector to $M$, and let $\rA_\nu$ denote the second fundamental form of $M$ in the direction of $\nu$
(i.e., $\rA_\nu = \nu \cdot \II$).   Then $M$ is austere if and only if for all $\nu$
\begin{equation}\label{fullAcond}
\rA_\nu^{(2j+1)} = \cos^2\!\theta(\nu) \rAtilde_\nu^{(2j-1)},\qquad j=0, \ldots, \lfloor k/2\rfloor,
\end{equation}
where (i) $\theta(\nu)$ is the angle between $TM$ and $\rJ\nu$ (where $\rJ$ is the ambient complex structure), (ii) $\rAtilde$ is the restriction of $\rA$ to the subspace of $TM$ orthogonal to $\rJ \nu$,
and (iii) the superscripted index indicates the operation of taking the elementary
symmetric function of that degree in the eigenvalues of a matrix representing the quadratic form with respect to an orthonormal basis.
(When the superscript is nonpositive or larger than the size of the matrix representative, the symmetric function in \eqref{fullAcond} is replaced by zero.)
\end{theorem}

When $M$ is a hypersurface these conditions become simpler: there is only one unit normal vector (up to a minus sign),
so $\rA_\nu$ is replaced by the shape operator $\rA$; the subspace orthogonal
to $\rJ \nu$ belongs to the holomorphic distribution $\calH$ on $M$; and  $\theta=0$, so that the above conditions become
\begin{equation}\label{hypercond}
\rA^{(2j+1)} = \rAtilde^{(2j-1)}, \qquad j=0, \ldots, n-1,
\end{equation}
where $\rAtilde$ denotes the restrction of $\rA$ to $\calH$.

Recall that the Fubini-Study metric on $\CP^n$ is such that the Hopf fibration $\pi:S^{2n+1} \to \CP^n$ is
a Riemannian submersion, with respect to the round metric on the sphere.  (We take the sphere to have radius
one, so that $\CP^n$ has holomorphic sectional curvature equal to 4.)  The simplified conditions \eqref{hypercond}
let us relate austere hypersurfaces in $\CP^n$ with those in $S^{2n+1}$:

\begin{cor}\label{sphereit} A hypersurface $M \subset \CP^n$ is austere if and only if $\Mhat = \pi^{-1}(M)$ is austere in $S^{2n+1}$.
\end{cor}
\begin{proof} We compute using moving frames along $M$ and $\Mhat$.  We will say that 
a orthonormal moving frame $(\ve_1, \ldots, \ve_{2n})$ is {\em unitary} if $\rJ \ve_1 = \ve_2$, $\rJ \ve_3 = \ve_4$, etc.,
and we let $\F$ denote the bundle of unitary frames on $\CP^n$.
We'll say that a unitary frame along hypersurface $M$ is {\em adapted} if $\ve_{2n}$ is normal to $M$.

Suppose $f:M \to \F\vert_M$ is a local section giving an adapted unitary frame.  
Then the connection forms on $\F$ satisfy $f^* \w^{2n}_i = A_{ij} \wo^j$, 
where $\wo^j$ denotes the 1-forms of the dual coframe on $M$ (for $1\le i,j \le 2n-1$) and 
$A_{ij}$ are the components of the shape operator of $M$ in this coframe.

Now take an orthonormal frame $\ehat_0, \ldots, \ehat_{2n}$ along $\Mhat$ such that $\ehat_0$ is tangent
to the Hopf fibers (and in fact is $\ri$ times the position vector of the point $z \in \Mhat \subset S^{2n+1} \subset \C^{n+1}$)
and $\pi_* \ehat_\alpha = \ve_\alpha$ for $1\le \alpha \le 2n$.  Then one easily computes (using the relationships between
the connection forms of the two moving frames, as laid out in \S1 of \cite{HopfHN}) that relative to the basis $(\ehat_0, \ldots, \ehat_{2n-1})$
for the tangent space to $\Mhat$,
the shape operator of $\Mhat$ is represented by
$$\Ahat = 
\left(\begin{array}{cccc}
0 & 0 & \ldots & 1 \\
0 & \multicolumn{2}{c}{\multirow{2}{*}{\scalebox{1}{$\Atilde$}}}& \multirow{2}{*}{\scalebox{1}{$v$}}\\
\vdots & & & \\
1 & \multicolumn{2}{c}{v^t} & \alpha
\end{array}\right),
\text{ where } A = \begin{pmatrix} \Atilde & v \\ v^t & \alpha \end{pmatrix}.
$$
Here, we use $\Atilde$ to denote the result of omitting the last row and column of $A$, yielding
a matrix representing the restriction of the second fundamental form to $\calH$.  Then one computes
$$\det(\lambda I - \Ahat) = \lambda^{2n} - A^{(1)} \lambda^{2n-1} + (A^{(2)}-1) \lambda^{2n-2} - (A^{(3)}-\Atilde^{(1)}) \lambda^{2n-3} + \ldots,$$
so that the odd-degree symmetric functions of the eigenvalues of $\Ahat$ vanish exactly when $M$ satisfies the austere condition.
\end{proof}

Austere hypersurfaces remain unclassified, even in low-dimensional spaces such as $\CP^2$.
It is unlikely that Corollary \ref{sphereit} extends to general codimension, since the austere condition in $\CP^n$
in a given normal direction $\nu$ involves the angle between $\rJ \nu$ and the tangent space, while the austere condition
in the sphere is the same in all normal directions.  The result doesn't necessarily help us classify austere hypersurfaces
in $\CP^n$ either, since the corresponding problem in $S^{2n+1}$ for $n\ge 2$ is unsolved, and only those hypersurfaces that are unions of
Hopf fibers will descend to $\CP^n$.

\section{Stark Hypersurfaces and their Shape Operators}
In order to sharpen the austere condition, and construct non-trivial examples, we make the following

\begin{defn}
A hypersurface $M \subset\CP^{n+1}$ is {\bf stark} if $M$ is austere and $\rAtilde$
is compatible with the complex structure, i.e., $\rAtilde(\rJ X, \rJ Y) = \rAtilde(X,Y)$ for $X,Y \in \calH$.
\end{defn}
If $M$ is stark, then the right-hand side of \eqref{hypercond} is zero, so that $M$ also satisfies
the Euclidean austere conditions $\rA^{(2j+1)}=0$.  (The name `stark' was chosen over more awkward
terms such as `doubly-austere'.)  For convenience, we now take the dimension of $M$ to be $2n+1$.

In general, oriented hypersurfaces in complex space forms are endowed with two natural tensor fields: the structure vector $W = -\rJ \nu$ and
the restriction $\varphi$ of the complex structure to the tangent space.
(Thus, the image of $\varphi$ is the distribution $\calH$ and its kernel is spanned by $W$.)
In what follows, we will need two different choices of orthonormal basis $\frakB=(\ve_1, \ldots, \ve_{2n+1})$ for $TM$
which differ in how $\varphi$ is represented.
We'll say $\frakB$ is a {\em standard basis} if $\ve_{2n+1}=W$ and $\varphi(\ve_j) = \ve_{j+n}$ for $1\le j \le n$, so that
$$[\varphi]_{\frakB} = J_n := \begin{pmatrix} 0 & -I_n & 0 \\ I_n & 0 & 0 \\ 0 & 0 & 0 \end{pmatrix}.$$
(We use square brackets to denote the matrix representative, with respect to a particular basis, for a linear
transformation or a quadratic form on the tangent space.)
We'll say $\frakB$ is a {\em split basis} if it is adapted to a splitting of $\calH$ into $\varphi$-invariant subspaces, so that
\begin{equation}\label{splitphi}
[\varphi]_{\frakB} = J_{k,\ell} := \begin{pmatrix} 0 & -I_k & 0  & 0 & 0 \\ I_k & 0 & 0 & 0 & 0 \\ 0 & 0 & 0 & -I_\ell &  0  \\ 0 & 0 & I_\ell & 0 & 0 \\ 0 & 0 & 0 & 0 & 0\end{pmatrix}, \qquad k+\ell = n.
\end{equation}

\begin{prop}\label{nonHopf} No stark hypersurface is Hopf (i.e., has $W$ as a principal vector).
\end{prop}
\begin{proof}
Suppose that $M$ is Hopf.   The stark conditions $\tr \rA = \tr \rAtilde = 0$ imply that $\langle \rA W, W\rangle = 0$,
so that $M$ is pseudo-Einstein.  Thus, with respect to a standard basis,
$$[\rA]_{\frakB} = \begin{pmatrix} P & Q & 0 \\ Q & -P & 0 \\ 0 & 0 & 0 \end{pmatrix},$$
where $P$ and $Q$ are $n \times n$ symmetric.  
However, in this case $\rA \varphi + \varphi \rA = 0$, and it is known that such hypersurfaces do not exist
(see Cor. 2.12 in \cite{RyanSurvey}).
\end{proof}

In the rest of this section, we will determine matrix representatives for shape operators
 satisfying the stark conditions.

\begin{prop}\label{algprop}  Let $M \subset \CP^{n+1}$ be a stark hypersurface with shape operator $\rA$.
Then at every point $m \in M$, there is either a choice of standard basis such that
\begin{equation}\label{irredform}
[\rA]_{\frakB} = \begin{pmatrix} 0 & S & d \\ S & 0 & 0 \\ d^t & 0 & 0 \end{pmatrix},
\end{equation}
where $S$ is $n\times n$ symmetric and $d$ is a nonzero vector of length $n$,
or a choice of split basis such that
\begin{equation}\label{redform}[\rA]_{\frakB} = \begin{pmatrix} P & Q & 0 & 0 & 0 \\ Q & -P & 0 & 0 & 0 \\ 0 & 0 & 0 & S & d \\ 0 & 0 & S & 0 & 0 \\
0 & 0 & d^t & 0 & 0  \end{pmatrix},
\end{equation}
where $P, Q$ are $k\times k$ symmetric for $k>0$, $S$ is $\ell \times \ell$ symmetric and $d$ is a vector of length $\ell$.
\end{prop}
\noindent
(The two cases are distinguished by whether or not the shape operator preserves a nonzero $\rJ$-invariant subspace of $\calH$.)
\begin{remark} It is easy to see that if the shape operator of $M$ takes either of the above forms, then $M$ is stark.
For example, in the case of \eqref{irredform}, one can check that $R [\rA] + [\rA] R=0$ where $R = \begin{pmatrix} -I_n & 0 \\ 0 & I_{n+1}\end{pmatrix}$; then the eigenvalues of $[\rA]$ are balanced around zero (see, e.g., Example 1 in \cite{Baustere}).
The balanced eigenvalue condition also holds for the matrix \eqref{redform} because it holds individually for the top left $2k\times 2k$ and
bottom right $(2\ell+1) \times (2\ell+1)$ blocks.
\end{remark}

\begin{proof}[Proof of Prop. \ref{algprop}] The stark conditions imply that for a standard basis $\frakB$,
$$[\rA]_{\frakB} = \begin{pmatrix} P & Q & b \\ Q & -P & c \\ b^t & c^t & 0 \end{pmatrix},$$
where $P,Q$ are $n\times n$ symmetric and $b,c$ vectors of length $n$.  If $b=c=0$ then we fall into the second case
(with $\ell=0$)
and we are done.  (However, a hypersurface with this shape operator would be Hopf, and by Prop. \ref{nonHopf} these
do not exist.)  Otherwise, we use the $U_n$ freedom to modify the standard basis so that $c=0$ and
$b=(\beta, 0,\ldots, 0)^t$ for $\beta > 0$.  (Here, $U_n$ denotes the subgroup 
of $O(2n+1)$ that commutes with $J_n$.)
Then, expanding $\det(\lambda I - A)$ and
taking coefficients of even powers of $\lambda$ shows that the $(2n-1) \times (2n-1)$ submatrix
$$
\left( \begin{array}{ccc} \Phat & q_1 & \Qhat \\ q_1^t & \multicolumn{2}{c}{\multirow{2}{*}{\scalebox{1}{$-P$} } }  \\ \Qhat & &  \end{array}\right),
\quad \text{where} \ P = \begin{pmatrix} p_{11} & p_1 \\ p_1^t & \Phat\end{pmatrix} \text{ and } Q = \begin{pmatrix} q_{11} & q_1 \\ q_1^t & \Qhat \end{pmatrix}$$
must be austere.  In particular, taking the trace implies that $p_{11}=0$.

We proceed by induction on $n$.  If $n=1$ then $P=(p_{11})$ is zero and we are done.
Otherwise, let $V' \subset T_m M$ denote the subspace orthogonal to $\{ W, \rA W\}$, let
$\pi_{V'}$ denote orthogonal projection onto $V'$, and let $\psi = \pi_{V'} \circ \varphi$ and
$\rA' = \pi_{V'} \circ \rA$ be restricted to $V'$.  Then
$$[\rA']_{\frakB} = \begin{pmatrix}  \Phat & q_1 & \Qhat \\ q_1^t & 0 & - p_1^t \\ \Qhat & -p_1 & -\Phat \end{pmatrix}, \qquad
[\psi]_{\frakB} = \begin{pmatrix}  0 & 0 & -I_{n-1} \\ 0 & 0 & 0 \\ I_{n-1} & 0 & 0 \end{pmatrix}.
$$
(By abuse of notation, we take the members $(\ve_2,\ldots, \ve_{2n})$ of $\frakB$ as a basis for $V'$.)

\medskip
\noindent {\bf Case (i).}
Suppose that $V'$ contains no proper subspace that is both $\rA'$- and $\psi$-invariant.
Let
$$R_1 = \begin{pmatrix} I_{n-1} & 0 & 0 \\ 0 & 0 & I_{n-1} \\ 0 & 1 & 0\end{pmatrix},
\text{ so that }R_1 [\rA']_{\frakB} R_1^t = \begin{pmatrix} \Phat & \Qhat & q_1 \\ \Qhat & -\Phat & -p_1 \\ q_1^t & -p_1^t & 0 \end{pmatrix}
$$
and $R_1 [\psi]_{\frakB} R_1^t = J_{n-1}$.  

By induction there exists a basis $\hat\frakB$ for $V'$ with respect to which $\rA'$ takes the form \eqref{irredform}
while $\psi$ is represented by $J_{n-1}$.
In other words, there exists a $G\in U_{n-1}$ such that
\begin{equation}\label{AprimeIrred}
[\rA']_{\hat\frakB} = G R_1 [\rA']_{\frakB} R_1^t G^t = \begin{pmatrix} 0 & S & d \\ S & 0 & 0 \\ d^t & 0 & 0 \end{pmatrix}
\end{equation}
where $S$ and $d$ have size $n-1$.  We will use
$$H = \begin{pmatrix} 1 & 0 & 0 \\ 0 & GR_1 & 0 \\ 0 & 0 & 1\end{pmatrix}$$
to change basis so as to put $\rA$ in the form \eqref{irredform}.  Noting that
$$G = \begin{pmatrix} K & -L & 0 \\ L & K & 0 \\ 0 & 0 & 1\end{pmatrix},$$
where $K,L$ are square matrices of size $n-1$, we compute that
\begin{equation}\label{bigHconj}
H [\rA]_{\frakB} H^t = \begin{pmatrix}
0 & (K p_1 - L q_1)^t & * & q_{11} & \beta \\
* & X & * & K p_1 +L q_1 & 0 \\
* & * & Y & L q_1 - K p_1 & 0 \\
* & * & * & 0 & 0 \\
* & 0 & 0 & 0 & 0 \end{pmatrix},
\end{equation}
where
\begin{align*}
X &= K \Qhat L^t + L \Qhat K^t + L \Phat L^t - K \Phat K^t,\\
Y &= K \Phat K^t - L \Phat L^t - L \Qhat K^t - K \Qhat L^t,
\end{align*}
and the $*$'s in \eqref{bigHconj} denote blocks that are either irrelevant or determined by symmetry.

By comparing the central $3\times3$ set of blocks on the  right-hand side of \eqref{bigHconj} with the the right-hand side of \eqref{AprimeIrred},
we see that $X=0$, $Y=0$ and $L q_1 - K p_1=0$.  Now \eqref{bigHconj}
has the desired form, but we also must put $\varphi$ in the correct form.  If we let
$$R_2 = \begin{pmatrix} 0 & I_{n-1} & 0 \\ 1 & 0 & 0 \\ 0 & 0 & I_{n-1} \end{pmatrix}$$
then $R_2 H [\phi]_{\frakB} H^t R_2^t = J_n$, while $R_2 H [\rA]_{\frakB} H^t R_2^t$ is still of the form \eqref{irredform}.

\medskip
\noindent{\bf Case (ii).} In this case, $V'$ contains a subspace of dimension $2k>0$
that is $\rA'$- and $\psi$-invariant.  Let $\ell=n-1-k$.
We first change the basis for $V'$ to one where $[\psi]$ has the split form; we do this
using the change of basis matrix
$$R_1 =
\begin{pmatrix} I_k & 0 & 0 & 0 & 0 \\ 0 & 0 & 0 & I_k & 0 \\ 0 & I_\ell & 0 & 0 & 0\\ 0 & 0 & 0 & 0 & I_\ell \\ 0 & 0 & 1 & 0 & 0\end{pmatrix},$$
which satisfies $R_1 [\psi]_{\frakB} R_1^t=J_{k,\ell}$.

Let $U_{k,\ell} \subset O(2n-1)$ denote the subgroup that commutes with $J_{k,\ell}$.
By induction, there exists a matrix $G \in U_{k,\ell}$ such that $G R_1 [\rA']_{\frakB} R_1^t G^t$ has the form \eqref{redform}.
But since $[\rA']_{\frakB}$ is a submatrix of $[\rA]_{\frakB}$ obtained by omitting the first row and column, before we can put $[\rA]$ in a similar form we must move the $2k\times 2k$
block in \eqref{redform} into the upper left corner of the larger matrix.
Taking
$$G = \begin{pmatrix} D & -E & 0 & 0 & 0 \\ E & D & 0 & 0 & 0 \\ 0 & 0 & K & -L & 0 \\ 0 & 0 & L & K & 0 \\ 0 & 0 & 0 & 0 & 1 \end{pmatrix},$$
where $D,E$ are $k \times k$ and $K,L$ are $\ell\times \ell$, we now use the following matrix to change basis on $V$:
$$H = \begin{pmatrix} 0 & I_{2k} &0 \\ 1 & 0 & 0 \\ 0 & 0 & I_{2\ell+2} \end{pmatrix}
\begin{pmatrix} 1 & 0 & 0 \\ 0 & G R_1 & 0 \\ 0 & 0 & 1 \end{pmatrix}
=  \begin{pmatrix} 0 & W & 0 & \tilde{W} & 0 \\ 1 & 0 & 0 & 0 & 0 \\ 0 & Z & 0 & \tilde{Z} & 0 \\ 0 & 0 & 1 & 0 & 0 \\ 0 & 0 & 0 & 0 & 1 \end{pmatrix}
$$
where
$$W_{2k \times (n-1)} = \begin{pmatrix}D & 0 \\ E & 0 \end{pmatrix}, \quad \tilde{W} = J_k W, \quad
Z_{2\ell\times (n-1)} = \begin{pmatrix}0 & K \\ 0 & L \end{pmatrix}, \quad \tilde{Z} = J_\ell Z.
$$
Using the induction hypothesis, we equate the result of conjugating $[\rA']_{\frakB}$ by $GR_1$ with a $(2n-1) \times (2n-1)$ matrix with block
form \eqref{redform}, and we deduce that
\begin{align*}
(W \Phat + \tilde{W} \Qhat) Z^t + (W \Qhat - \tilde{W}\Phat) \tilde{Z}^t &= 0, & W q_1 - \tilde{W} p_1 &=0,\\
(Z \Phat + \tilde{Z} \Qhat) Z^t + (Z \Qhat - \tilde{Z}\Phat) \tilde{Z}^t &= \begin{pmatrix} 0 & S \\ S & 0 \end{pmatrix},
& Z q_1 - \tilde{Z} p_1 &= \begin{pmatrix} d \\ 0 \end{pmatrix}.
\end{align*}
Using these, we compute that $H [\rA]_{\frakB} H^t$ has the form \eqref{redform} with $\ell$ replaced by $\ell+1$.
However, in order to ensure that $\varphi$ is represented by $J_{k,\ell+1}$,
we need to conjugate once more, by the matrix
$$R_2 = \begin{pmatrix} I_{2k} & 0 & 0 & 0 \\ 0 & 0 & I_{\ell} & 0 \\ 0 & 1 & 0 & 0 \\ 0 & 0 & 0 & I_{\ell+2} \end{pmatrix},$$
which preserves the form of $H [\rA]_{\frakB} H^t$.
\end{proof}

\section{Geometry of Stark 3-folds}
In this section we will use moving frames to investigate the geometry of stark hypersurfaces in $\CP^2$.
In particular, we will show that these hypersurfaces carry two perpendicular foliations, by helices and
by open subsets of totally geodesic $\RP^2$'s.  As discussed in the next section, we expect that these features generalize
to higher dimensions.

We will define an exterior
differential system whose integral manifolds are in one-to-one correspondence with local unitary frames
along stark hypersurfaces in $\CP^2$.  Using this system, we will show that these hypersurfaces are essential
described by a system of ordinary differential equations.

Again, let $\F$ be the unitary frame bundle of $\CP^2$.  This is a $U(2)$-subbundle of the full orthonormal frame bundle,
to which the canonical forms $\w^a$ and connection forms $\w^a_b$ restrict to satisfy the usual structure equations,
with the additional
relations $\w^1_3 = \w^2_4$ and $\w^3_2=\w^1_4$.  The curvature 2-forms on $\F$ are
\begin{align*}
\Psi^1_2 &= 4\w^1 \wedge \w^2 + 2\w^3 \wedge \w^4, & \Psi^1_3 = \Psi^2_4 &= \w^1 \wedge \w^3 + \w^2 \wedge \w^4, \\
\Psi^3_4 &= 2 \w^1 \wedge \w^2 + 4 \w^3 \wedge \w^4,& \Psi^3_2=\Psi^1_4 &= \w^1 \wedge \w^4 - \w^2 \wedge \w^3.
\end{align*}
(A careful derivation of these curvature forms is given in \S1 of \cite{HopfHN}.)

Using the reproducing property of the canonical forms, we see that adapted unitary frames along a hypersurface $M$
are precisely the sections of $\F\vert_M$ along which $\w^4$ vanishes.  Furthermore, if $A_{ij}$
are components of the shape operator with respect to the standard basis $(\ve_1,\ve_2,\ve_3)$, then
the 1-forms $\w^4_j - A_{ij} \w^j$ also vanish along this section.  Given these facts, we
define an exterior differential system $\I$ on $\F\times \R^2$ generated by the 1-forms
$$\theta_0 :=\w^4,\quad
\theta_1:=-\w^4_1 + \mu \w^2 + \beta \w^3,\quad
\theta_2 := -\w^4_2 + \mu \w^1, \quad
\theta_3:=-\w^4_3 +\beta \w^1.
$$
(The components $\beta,\mu$ of the shape operator are introduced as extra variables, and are coordinates
on the $\R^2$ factor.)
Then integral 3-folds of $\I$ 
are in one-to-one correspondence with stark hypersurfaces equipped with a standard moving frame with respect to
which the shape operator has the form predicted by Prop. \ref{algprop}, i.e.,
$$[\rA] = \begin{pmatrix} 0 & \mu & \beta \\ \mu & 0 & 0 \\ \beta & 0 & 0 \end{pmatrix}.$$
(We will assume that all integral submanifolds of $\I$ and its prolongations satisfy the usual independence condition $\w^1 \wedge \w^2 \wedge \w^3 \ne 0$.)
By Prop. \ref{nonHopf}, the set of points where $\beta$ vanishes has empty interior; therefore,
we will restrict our attention to the open subset where $\beta \ne 0$.

We begin to determine the solution space of $\I$ by calculating the system 2-forms (thus completing
a set of algebraic generators for $\I$).  We compute $d\theta_0 \equiv 0$  and
\begin{equation}\label{the2forms}
\left.
\begin{aligned}
d\theta_1 &\equiv -2\mu\, \w^2_1 \wedge \w^1+ \pi_1 \wedge \w^2 + \pi_2 \wedge \w^3,\\
d\theta_2 &\equiv \pi_1 \wedge \w^1 + \w^2_1 \wedge (2\mu \w^2+\beta \w^3 ),\\
d\theta_3 &\equiv \pi_2 \wedge \w^1 + \beta\, \w^2_1 \wedge \w^2,
\end{aligned}\right\} \mod \theta_0, \ldots \theta_3,
\end{equation}
where
$$
\pi_1 := d\mu + \beta\mu\, \w^2 + (\beta^2 -\mu^2 -1) \w^3, \qquad
\pi_2 := d\beta - 2(\mu^2 + 1) \w^2-3\beta \mu\, \w^3.
$$

\begin{prop}\label{ruledM} Any stark hypersurface $M \subset \CP^2$ is ruled by totally geodesic surfaces
tangent to the 2-dimensional nullspace of $\rA$ that contains the structure vector.  These surfaces are open subsets of copies of $\RP^2$.
\end{prop}
\begin{proof} Let $N \subset \F \times \R^2$ be an integral 3-fold of $\I$ corresponding
to a unitary frame along $M$ for which $\beta \ne 0$.  Since
$$d\w^1 = \w^2_1 \wedge \w^2 +\w^4_2 \wedge \w^3 +\w^4_1 \wedge \w^4 \equiv 0 \mod \I, \w^1$$
(using the 1-form $\theta_2$ and the last 2-form in \eqref{the2forms}),
then $\w^1$ restricts to be integrable on $N$.  Thus, its pullback  to $M$ annihilates
an integrable distribution spanned by $\ve_2, \ve_3$.

Let $\Sigma \subset M$ be a surface tangent to this distribution.  Since $\rJ \ve_3 = \ve_4$ and $\rJ \ve_2 = -\ve_1$
are normal to $\Sigma$, $\Sigma$ is totally real.  Since $\w^4_2, \w^4_3 \equiv 0 \mod \I, \w^1$ then
the second fundamental form of $\Sigma$ in the direction of $\ve_4$ is zero; since $\w^1_3 = \w^4_2$ it remains
only to check that $\w^2_1 \equiv 0 \mod \I, \w^1$ to confirm that $\Sigma$ is totally geodesic.

Because $N$, satisfies the independence condition, $\w^2_1$ must be equal to a linear
combination of $\w^1, \w^2, \w^3$ along $N$.  Applying the Cartan Lemma to the vanishing
of the last 2-form in \eqref{the2forms} implies that $\w^2_1$ must lie in the span of $\w^1$ and $\w^2$,
but doing the same for the second 2-form in \eqref{the2forms} shows that
$\w^2_1$ must lie in the span of $\w^1$ and $\beta \w^3 + 2\mu \w^2$.  Thus $\w^2_1$ must
restrict to $N$ to be a multiple of $\w^1$.

Because $\Sigma$ is totally geodesic and totally real, it is congruent to an open set of
a $\RP^2 \subset \CP^2$ (see Theorem 4 in \cite{Wo}).
\end{proof}

In Prop. \ref{ruledM} we showed that $\w^2_1$ must restrict to
be a multiple of $\w^1$ on any integral element of $\I$ (i.e., the tangent space to
an integral 3-fold satisfying the independence condition).  More formally, we define the {\em prolongation} of $\I$
as the system of differential forms that vanish along these integral elements.  In this case,
we introduce a new coordinate $\kappa$ and define additional 1-forms
\begin{align*}
\theta_4 &:= \w^2_1 - \kappa\, \w^1, \\
\theta_5 &:= \pi_1 - \kappa(2\mu \w^2+\beta \w^3 )\\
\theta_6 &:= \pi_2 - \beta\kappa\, \w^2
\end{align*}
on $\F\times \R^3$; the prolongation $\I'$ is then generated by $\theta_0, \ldots, \theta_6$.

\begin{prop}\label{helical} Let $M$ be a stark hypersurface in $\CP^2$ and let $\gamma$ be a trajectory in $M$
orthogonal to the rulings of Prop. \ref{ruledM}.  Then $\gamma$ is a helix.
Moreover, if $\gamma$ closes up smoothly at length $L$, then so do all other such curves in $M$.
\end{prop}
\begin{proof}
Take a unitary frame $(\ve_1, \ldots, \ve_4)$ along $M$; as noted above, this frame (coupled with the values of $\beta, \mu, \kappa$)
gives a section of $\F \times \R^3$ that is an integral $N$ of $\I'$.
At the same time, the unitary frame  restricts to be a Frenet-type frame along $\gamma$.
Because the connection forms on $\F$ have the property that $\nabla \ve_a = \ve_b \otimes \w^b_a$ for any section,
we can calculate the covariant derivatives of the frame vectors along $\gamma$ by computing the $\w^1$-component
of the corresponding connection form restricted to $N$.
Using $D/ds$ to denote covariant derivative with respect to arclength along $\gamma$, we obtain
\begin{equation}\label{helixF}\dfrac{D\ve_1}{ds} = \kappa \ve_2 + \mu \ve_3, \quad \dfrac{D\ve_2}{ds} =-\kappa \ve_1 + \mu \ve_4,
\quad
\dfrac{D\ve_3}{ds} = -\mu \ve_1 + \beta \ve_4, \quad \dfrac{D\ve_4}{ds} = -\mu \ve_2 - \beta \ve_3.
\end{equation}
Because $\theta_5, \theta_6$ vanish along $N$, $d\mu$ and $d\beta$ have no $\w^1$-component, and thus
$\beta, \mu$ are constant along $\gamma$.  To show $\gamma$ is a helix, we must show that all its remaining
curvature $\kappa$ is constant along $\gamma$.

Differentiating the new 1-forms modulo themselves yields
$$
\left.\begin{aligned}
d\theta_4 &\equiv \w^1 \wedge \pi_3,\\
d\theta_5 &\equiv (2\mu \w^2+\beta \w^3 ) \wedge \pi_3,\\
d\theta_6 &\equiv \beta\, \w^2 \wedge \pi_3,
\end{aligned}
\right\} \mod \theta_0, \ldots, \theta_6,
$$
where
$$\pi_3 := d\kappa + (2\mu^2 - \kappa^2-4 )\w^2 + \mu(2\beta - \kappa)\w^3.$$
The vanishing of each of these 2-forms implies that, on an integral 3-fold,
$\pi_3$ must be a multiple of three linearly independent 1-forms.  Thus, $\pi_3$
vanishes on any integral 3-fold of $\I'$.  In particular, $\kappa$ is constant
along $\gamma$.

The criteria under which a helix such as $\gamma$ is smoothly closed are laid out in \cite{BGS} (see \S6,7).
Let $X(s)$ be a lift of $\gamma$ into $S^5$ (relative to the Hopf fibration), and let $E_1(s), E_3(s)$ be horizontal lifts along $X$ of the frame vectors
$\ve_1, \ve_3$.  Then $F(s)=(X,E_1,E_3)$ takes value in $U(3)$ and satisfies the constant-coefficient system
$$\dfrac{dF}{ds} = F K, \qquad \text{where } K = \begin{pmatrix} 0 & -1 & 0 \\ 1 & \ri \kappa & -\mu \\ 0 & \mu & \ri \beta \end{pmatrix}.$$
Then $\gamma$ is smoothly closed at length $L$ if and only if $F(s+L) = e^{\ri \theta I} F(s)$ for some $\theta\in \R$, i.e., the projection of $F(s)$ into the
quotient $SU(3) =  U(3)/S^1$ is $L$-periodic.  This in turn is equivalent to the eigenvalues $\epsilon_j$ of $K_0$ (the traceless part of $K$)
being integer multiplies of $2\pi \ri/L$.  Since these imaginary eigenvalues sum to zero, it is necessary
and sufficient that the ratio of any two of them be rational.  (For example, if $\epsilon_1/\epsilon_2 = n_1/n_2$ in lowest terms,
then $\gamma$ is closed at length $L = n_1 (2\pi \ri /\epsilon_1)$.)  The characteristic polynomial of $K_0$ is
$$\det(\lambda I - K_0) = \det( \lambda I - (K - \tfrac13 (\tr K)I)) = \lambda^3 +A \lambda + \ri B,$$
where
$$A = \mu^2 + \tfrac13(\beta^2 - \beta \kappa + \kappa^2) + 1,\qquad
B = \tfrac1{27}(2\beta^3 -3\beta^2\kappa - 3\beta \kappa^2 + 2\kappa^3) + \tfrac13(\mu^2(\beta+\kappa)+\kappa-2\beta).
$$
By replacing $K_0$ by a diagonal matrix with entries $(\epsilon_1, \epsilon_2, -\epsilon_1-\epsilon_2)$, we can calculate that the ratio $A^3/B^2$ is a rational function of $\epsilon_1/\epsilon_2$.  Using the values of the differentials
\begin{align*}
d\beta &= (\beta\kappa + 2\mu^2+2)\w^2 + 3\beta\mu \w^3, \\
d\mu &= \mu(2\kappa-\beta)\w^2+(\mu^2+\beta\kappa-\beta^2+1)\w^3\\
d\kappa &=(\kappa^2-2\mu^2+4)\w^2 +\mu(\kappa-2\beta)\w^3, \\
\end{align*}
(which are implied by the vanishing of $\theta_5, \theta_6$ and $\pi_3$ respectively along $N$)
we compute that $A^3/B^2$ is constant on $M$.  Thus, if the rationality condition is satisfied for a particular
helix $\gamma$, then it is satisfied for all helices, and all smoothly close up at length $L$.
\end{proof}

\begin{prop} Stark hypersurfaces in $\CP^2$ comprise a 3-parameter family, modulo isometries.  A unique
such hypersurface $M$ exists through any given point $p \in \CP^2$, given a choice of unitary frame $(\ve_1, \ldots, \ve_4)$
(where $\ve_4$ is to be the hyperface normal) at $p$ and prescribed initial values of $\beta, \kappa, \mu$ (with $\beta \ne 0$).
\end{prop}
\begin{proof}
Let $\J$ be the Pfaffian system on $\F \times \R^3$ generated by adjoining the 1-form $\pi_3$ to $\I'$;
as noted in the proof of Prop. \ref{helical}, the any integral 3-fold of $\I'$ is also an integral of $\J$.
It is easy to check that the system $\J$ is Frobenius, i.e., the exterior derivative
of each generator 1-form $\theta_0, \ldots, \theta_6, \pi_3$ can be expressed as a sum of wedge products involving
those 1-forms.  By the Frobenius Theorem, there exists a unique maximal integral 3-fold through each
point of the 11-dimensional manifold $\F \times \R^3$.  Moreover, the data listed in the last sentence of the proposition is precisely
enough to determine a unique point in this space.   Since the isometry group $SU(3)$ of $\CP^2$ acts transitively
on choices of point $p$ and a unitary frame at $p$, the set of stark hypersurfaces modulo ambient isometries is 3-dimensional.
\end{proof}

The Frobenius Theorem implies that there are local coordinates in which the system $\J$ becomes a system
of total ordinary differential equations in 8 unknowns (same as the rank of $\J$).  In fact, solutions may be determined
up to congruence by solving a much smaller system of ODE, in a geometrically natural set of local coordinates.
For, at points where one can easily compute that the following pair of 1-forms are closed on $N$:
$$\beta^{-1/3}\w^2, \qquad \mu^{4/3} \w^2 + \beta^{2/3}\mu^{1/3} \w^3.$$
Thus, away from points where $\mu \ne 0$, we may introduce smooth local coordinates $x,y$
such that
\begin{equation}\label{deexy}
dx=\beta^{-1/3}\w^2,\qquad dy = \mu^{4/3} \w^2 + \beta^{2/3}\mu^{1/3} \w^3.
\end{equation}
Since these 1-forms have the same span as $\w^2, \w^3$, we may express the differentials
of $\beta,\mu$ and $\kappa$ in terms of them, yielding a system of total differential
equations for these variables as functions of $x$ and $y$.  These equations 
may be expressed in rational form if we make the change of dependent variables
$$t=\dfrac{\kappa}{\beta}, \qquad u = \left(\dfrac{\mu}{\beta}\right)^{2/3}, \quad v = \beta^{2/3}.$$
Then we obtain the following system of total differential equations:
\begin{equation}\label{tuvsys}
\begin{aligned}
dt &= \dfrac{2(2-t)}{v}dx - 2u(t+1)dy,\\
du &= -\dfrac{2u}{v} dx +\dfrac{2}{3u}(t + v^{-3} -2u^3-1)dy,\\
dv &= \tfrac23(v^3(t-u^3) +2)dx + 2uv\,dy.
\end{aligned}
\end{equation}
We may rewrite the ODEs in the $y$-direction as a single third-order equation for $v$, which turns out to have the
following first integrals:
$$C := \tfrac13( (t-u^3)v^2- v^{-1}), \qquad D:=\tfrac13 v (t+1).$$
These are constant in the $y$-direction, and satisfy a simple ODE system in the $x$-direction:
\begin{equation}\label{CDode}
\dfrac{dC}{dx} = 4(C^2+D), \qquad \dfrac{dD}{dx} = 2(CD+1).
\end{equation}
The first integral $A^3/B^2$ can be expressed in terms of $C,D$ as
$$\dfrac{A^3}{B^2} = \dfrac{27(D^2-C)^3}{(2D^2-3CD-1)^2}.$$
Taken in reverse order, these equations enable us (in theory) to construct stark hypersurfaces in $\CP^2$
by a sequence of integrations:
\begin{enumerate}
\item Choose a constant value for $A^3/B^2$ which admits real values of $C$ and $D$, and solve the first order ODEs \eqref{CDode} for $C$ and $D$;
\item Replacing $t,u$ by their values in terms of $C,D$ and $v$, integrate the last equation in \eqref{tuvsys};
\item Determine an orthonormal coframe on the $xy$ domain by solving \eqref{deexy} for $\w^2, \w^3$;
\item Identify the $xy$ domain isometrically with an open subset $\Sigma$ of $\RP^2 \subset \CP^2$;
\item Integrate the Frenet equations \eqref{helixF} to produce helices through points of $\Sigma$,
and let $M$ be the union of these helices.
\end{enumerate}
\begin{remark*}
While these hypersurfaces are foliated by helices, they are different from the generalized helicoids of \cite{Baustere}.
The latter are austere submanifolds in Euclidean space, swept out by applying a 1-parameter family of screw motions to a $k$-dimensional subspace in $\R^{2k+1}$.
However, although each helix in $M$ is the orbit of a 1-parameter subgroup of $SU(3)$, the subgroup changes as we move across $\Sigma$.
Indeed, the conjugacy class of the subgroup is given by the values of $\epsilon_1, \epsilon_2$ and $-\epsilon_1-\epsilon_2$ (up to permutation),
and these vary along $\Sigma$ because neither $A$ nor $B$ is constant along $\Sigma$.
\end{remark*}

\section{Remarks on higher-dimensional examples}
\begin{defn} Following Prop. \ref{algprop}, we will say that a stark hypersurface is {\em reducible} if the distribution
$\calH$ contains a nonzero subspace that is both $\rA$- and $\rJ$-invariant, and {\em irreducible} otherwise.
\end{defn}

For an irreducible hypersurface $M\subset \CP^{n+1}$, Prop. \ref{algprop} asserts that at each point there
is a standard basis $(\ve_1, \ldots,  \ve_{2n+1})$ with respect to which the shape operator has the form
\eqref{irredform}.  It follows that we may define two orthogonal distributions
$$\calU = \{ \ve_1, \ldots, \ve_n\}, \qquad \calV = \{ \ve_{n+1}, \ldots, \ve_{2n+1}\}$$
which are null spaces for the shape operator.  (Note also that $\calV$ is totally real.)
In the case $n=1$ discussed in the last section, $\calV$ was tangent to a foliation of $M$ by totally geodesic surfaces,
and curves tangent to the 1-dimensional distribution $\calU$ were helices, i.e., orbits of a 1-parameter subgroup of $SU(3)$.
Accordingly, we make the following

\begin{conj}For an irreducible stark hypersurface,
\begin{itemize}
\item the distributions $\calU$ and $\calV$ are integrable;
\item the leaves tangent to $\calV$ are totally geodesic (hence, open subsets of $\RP^n$'s);
\item the leaves tangent to $\calU$ are orbits of $n$-dimensional tori in $SU(n+2)$.
\item Such hypersurfaces are determined by integrals of a Frobenius exterior differential system.
\end{itemize}
\end{conj}
\noindent
These conjectures have been verified in the cases $n=1,2,3$.

For a reducible hypersurface, there is at each point a split basis adapted to a $2k$-dimensional
$\rA$- and $\rJ$-invariant subspace $\calW \subset \calH$, with respect to which the shape operator has
the form \eqref{redform}.  Let $\calW^\perp$ denote its orthogonal complement (of dimension $2\ell+1$, where $k + \ell =n$) within $TM$.  We make the following

\begin{conj}For a reducible stark hypersurface,
\begin{itemize}
\item the distribution $\calW^\perp$ is integrable, but $\calW$ is not;
\item the leaves tangent to $\calW^\perp$ are irreducible stark hypersurfaces lying in totally geodesic copies
of $\CP^{\ell+1}$;
\item any two $\calW^\perp$-leaves are congruent to each other.
\item Such hypersurfaces are determined by integrals of a Frobenius exterior differential system.
\end{itemize}
\end{conj}
\noindent
These conjectures have been verified only in the case $k=\ell=1$.

\section*{Acknowledgements}
The author thanks Marianty Ionel, Tommy Murphy and Patrick Ryan for helpful discussions of this work.  He is
also grateful to Bogdan Suceava for encouraging its publication.


\begin{thebibliography}{99}
\def\tit{\bf}
\def\jou{\it}
\def\book{\sf}

\bibitem{BGS} M. Barros, O. Garay, D.A. Singer, {\tit Elasticae with constant slant in the complex projective plane},
{\jou Tohoku Math. J.} 51 (1999),  177--192.

\bibitem{Baustere}
R.L. Bryant, {\tit Some remarks on the geometry of austere manifolds},
{\jou Bol. Soc. Bras. Mat.} 21 (1991), 133--157.

\bibitem{DF}
M. Dajczer, L. Florit, {\tit A class of austere submanifolds},
Illinois J. Math. 45 (2001), 735--755.

\bibitem{HL} R. Harvey, H.B. Lawson, {\tit Calibrated Geometries}, {\jou Acta Math.} 148 (1982), 47--157.

\bibitem{II1} M. Ionel, T. Ivey, {\tit Austere Submanifolds of Dimension Four:
Examples and Maximal Types}, {\jou Illinois Math J.} 54 (2010), 713--746.

\bibitem{II2} ---, {\tit Ruled Austere Submanifolds of Dimension Four}, {\jou Differential Geometry
and its Applications} 30 (2012), 588--603.

\bibitem{II3} ---, {\tit Austere Submanifolds in Complex Projective Space}, to appear in {\jou Communications
in Analysis and Geometry}.

\bibitem{HopfHN} T. Ivey, {\tit A d'Alembert Formula for Hopf Hypersurfaces}, {\jou Results in Mathematics} 60 (2011), 293--309.


\bibitem{KM} S. Karigiannis, M. Min-Oo,
{\tit Calibrated subbundles in noncompact manifolds of special holonomy},
Ann. Global Anal. Geom.  28  (2005), 371--394.

\bibitem{RyanSurvey} R. Niebergall, P.J. Ryan, {\tit Real hypersurfaces in complex space forms} in {\book
Tight and taut submanifolds}, MSRI Publications 32, Cambridge University Press, 1997.


\bibitem{Stpaper} M. Stenzel, {\tit Ricci-flat metrics on the complexification of a compact rank one symmetric space},
{\jou Manuscripta Math.} 80  (1993), 151--163.


\bibitem{Wo} J.A. Wolf, {\tit Geodesic spheres in Grassmann manifolds} and {\tit Elliptic spaces in Grassmann manifolds},
{\jou Illinois Math. J.} 7 (1963), 425--462.

\end{thebibliography}
\end{document}